\documentclass[a4paper,11pt,reqno]{amsart}
\usepackage{amsmath,amsthm}
\usepackage{amssymb,latexsym}
\usepackage{amsaddr}
\usepackage{graphicx,subfig}
\usepackage{epsfig}
\usepackage{color}
\usepackage{setspace}

\usepackage{lineno}
\newcommand*\patchAmsMathEnvironmentForLineno[1]{%
  \expandafter\let\csname old#1\expandafter\endcsname\csname #1\endcsname
  \expandafter\let\csname oldend#1\expandafter\endcsname\csname end#1\endcsname
  \renewenvironment{#1}%
     {\linenomath\csname old#1\endcsname}%
     {\csname oldend#1\endcsname\endlinenomath}}% 
\newcommand*\patchBothAmsMathEnvironmentsForLineno[1]{%
  \patchAmsMathEnvironmentForLineno{#1}%
  \patchAmsMathEnvironmentForLineno{#1*}}%
\AtBeginDocument{%
\patchBothAmsMathEnvironmentsForLineno{equation}%
\patchBothAmsMathEnvironmentsForLineno{align}%
\patchBothAmsMathEnvironmentsForLineno{flalign}%
\patchBothAmsMathEnvironmentsForLineno{alignat}%
\patchBothAmsMathEnvironmentsForLineno{gather}%
\patchBothAmsMathEnvironmentsForLineno{multline}%
}

\newtheorem{thm}{Theorem}[section]
\newtheorem{dfn}[thm]{Definition}
\newtheorem{lma}[thm]{Lemma}
\newtheorem{cor}[thm]{Corollary}
\newtheorem{prp}[thm]{Proposition}
\newtheorem{exm}[thm]{Example}

\newtheorem{conj}[thm]{Conjecture}

\numberwithin{equation}{section}

\def\eps{\varepsilon}
\DeclareMathOperator*{\Dmon}{\Delta_{{\rm mon}}}
\DeclareMathOperator*{\dist}{{\rm dist}}

\title[Properly coloured Hamiltonian cycles in $K_{n}^{c}$]{Properly coloured Hamiltonian cycles in edge-coloured complete graphs}

\author{Allan Lo}
\address{School of Mathematics, University of Birmingham,\\Birmingham, B15 2TT, UK}
\email{s.a.lo@bham.ac.uk}
\thanks{The research leading to these results was supported by the  European Research Council
under the ERC Grant Agreement no. 258345.}
\date{\today}
\begin{document}

% \doublespacing

%\linenumbers

\begin{abstract}
Let $K_n^c$ be an edge-coloured complete graph on $n$ vertices. 
Let $\Dmon(K_n^c)$ denote the largest number of edges of the same colour incident with a vertex of $K_n^c$.
A properly coloured cycle is a cycle such that no two adjacent edges have the same colour.
In 1976, Bollob\'as and Erd\H{o}s~\cite{MR0411999} conjectured that every $K_n^c$ with $\Dmon(K_n^c) < \lfloor n/2 \rfloor$ contains a properly coloured Hamiltonian cycle.
In this paper, we show that for any $\eps > 0 $, there exists an integer $n_0$ such that every $K_n^c$ with $\Dmon(K_n^c) < (1/2 - \eps) n $ and $n \ge n_0$ contains a properly coloured Hamiltonian cycle.
This improves a result of Alon and Gutin~\cite{MR1610269}.
Hence, the conjecture of Bollob\'as and Erd\H{o}s is true asymptotically. 
\end{abstract}

\maketitle

\section{Introduction}
An \emph{edge-coloured graph} is a graph $G$ with an edge-colouring~$c$ of~$G$.
We say that $G$ is \emph{properly coloured} if no two adjacent edges of $G$ have the same colour.
If all edges have the same colour, then $G$ is \emph{monochromatic}.

Let $K_n^c$ be an edge-coloured complete graph on $n$ vertices.
Let $\Dmon(K_n^c)$ denote the maximum number of edges of the same colour incident with a vertex of $K_n^c$.
Equivalently, $\Dmon(K_n^c) = \max  \Delta(H)$ over all monochromatic subgraphs $H$ in $K_n^c$.
Daykin~\cite{MR0406846} asked whether there exists a constant $\mu$ such that every $K_n^c$ with $\Dmon(K_n^c) \le \mu n$ and $n \ge 3$ contains a properly coloured Hamiltonian cycle.
This question was answered independently by Bollob\'as and Erd\H{o}s~\cite{MR0411999} with $ \mu = 1/69$, and Chen and Daykin~\cite{MR0422070} with $\mu = 1/17$.
Bollob\'as and Erd\H{o}s proposed the following conjecture.
\begin{conj}[Bollob\'as and Erd\H{o}s~\cite{MR0411999}] \label{conj}
If $\Dmon(K_n^c) < \lfloor n/2 \rfloor $, then $K_n^c$ contains a properly coloured Hamiltonian cycle.
\end{conj}
Later, Shearer~\cite{MR523092} showed that $\Dmon(K_n^c) \le n /7$ is sufficient.
The best known bound on $\Dmon(K_n^c)$ was given by Alon and Gutin~\cite{MR1610269} where $\Dmon(K_n^c) < (1-1/\sqrt2 - o(1))n$.
On the other hand, Li, Wang and Zhou~\cite{LiWangZhou} showed that if $\Dmon(K_n^c) < \lfloor n/2 \rfloor$, then $K_n^c$ contains a properly coloured cycle of length at least $(n+2)/3+1$.

For the existence of a properly coloured Hamiltonian path, Barr~\cite{MR1670602} proved that $K_n^c$ containing no monochromatic triangle is a sufficient condition.
Note that there is no assumption on $\Dmon ( K_n^c )$.
A \emph{$2$-factor} is a spanning 2-regular graph. 
Bang-Jensen, Gutin and Yeo~\cite{MR1609957} showed that $K_n^c$ containing a properly coloured $2$-factor is also a sufficient condition.
\begin{thm}[Bang-Jensen, Gutin and Yeo~\cite{MR1609957}] \label{Bang}
%Let $K_n^c$ be an edge-coloured~$K_n$.
If $K_n^c$ contains a properly coloured $2$-factor, then $K_n^c$ contains a properly coloured Hamiltonian path.
\end{thm}
This result was later improved by Feng, Giesen, Guo, Gutin, Jensen and Rafiey~\cite{MR2270727}.
A graph $G$ is said to be a \emph{$1$-path-cycle} if $G$ is a vertex-disjoint union of at most one path $P$ and a number of cycles.
Note that a spanning $1$-path-cycle without any cycles is a Hamiltonian path, and a spanning $1$-path-cycle without a path is a $2$-factor.

\begin{thm}[Feng, Giesen, Guo, Gutin, Jensen and Rafiey~\cite{MR2270727}] \label{Feng}
Let $K_n^c$ be an edge-coloured $K_n$.
Then $K_n^c$ contains a properly coloured Hamiltonian path
if and only if $K_n^c$ contains a spanning properly coloured $1$-path-cycle.
\end{thm}

For a survey regarding properly coloured subgraphs in edge-coloured graphs, we recommend Chapter~16 of~\cite{MR2472389}.
In this paper, we prove that Conjecture~\ref{conj} is true asymptotically. 

\begin{thm} \label{PCHC}
For any $\eps> 0$, there exists an integer $N_0 = N_0(\eps)$ such that every $K_n^c$ with $n \ge N_0$ and $\Dmon(K_n^c) \le (1/2 - \eps)n $ contains a properly coloured Hamiltonian cycle.
\end{thm}

For an edge-coloured graph $G$ (not necessarily complete), the \textit{colour degree}~$d^c(v)$ of a vertex~$v$ is the number of different colours of edges incident to~$v$.
The \textit{minimum colour degree} $\delta^c(G)$ of an edge-coloured graph~$G$ is the minimum $d^c(v)$ over all vertices~$v$ in~$G$.
Li and Wang~\cite{MR2519172} proved that every edge-coloured graph $G$ contains a properly coloured path of length $2\delta^c(G)$ or a properly coloured cycle of length at least $2\delta^c(G)/3$.
In~\cite{Lo10}, the author improved their result by showing that $G$ contains a properly coloured path of length $2\delta^c(G)$ or a properly coloured cycle of length at least $\delta^c(G)+1$.
Furthermore, in \cite{LoDirac}, the author proved that every edge-coloured graph $G$ on $n$ vertices with $\delta^c(G) \ge (2/3 + \eps) n$ contains a properly coloured cycle of length $\ell$ for all $3 \le \ell \le n$ provided $\eps >0$ and $n$ is large enough.
Moreover, the bound on $\delta^c(G)$ is asymptotically best possible; that is, there exist edge-coloured graphs $G$ on $n$ vertices with $\delta^c(G) =  \lceil 2n/3 \rceil -1$, which does not contain a properly coloured Hamiltonian cycle.
Note that $\delta^c(K_n^c) + \Dmon(K_n^c) \le n$. 
Hence, Theorem~\ref{PCHC} implies the following corollary.

\begin{cor}
For any $\eps> 0$, there exists an integer $N_0 = N_0(\eps)$ such that every $K_n^c$ with $n \ge N_0$ and $\delta^c(K_n^c) \ge (1/2 + \eps)n $ contains a properly coloured Hamiltonian cycle.
\end{cor}

Now we outline the proof of Theorem~\ref{PCHC}, which involves two main steps.
In the first step, we find (by Lemma~\ref{lma:abscycle}) a small `absorbing cycle' $C$ such that for any properly coloured path $P$ with $V(C) \cap V(P) = \emptyset$ and $|P| \ge 4$, there exists a properly coloured cycle $C'$ with $V(C') = V(P) \cup V(C)$.
This step can be viewed as a properly edge-coloured version of the absorption technique introduced by R\"odl, Ruci\'{n}ski and Szemer\'{e}di~\cite{MR2399020}.
Since the original absorption technique did not consider edge-coloured graphs, several new ideas are needed for this generalisation. 
We believe that there is further potential for this adaptation of the absorption technique.
For instance, a similar argument was also used in~\cite{LoDirac}.
In the second step, we remove the vertices of the small absorbing cycle $C$ from $K_n^c$ and let $K_{n'}^c$ be the resulting graph.
Since $C$ is small, we may assume that $\Dmon(K_{n'}^c) \le (1- \eps') n'$ for some small $\eps' >0$.
Next, we find a properly coloured $2$-factor in $K_{n'}^c$ using Lemma~\ref{lma:2factor}.
Hence, Theorem~\ref{Bang} implies that there exists a properly coloured Hamiltonian path $P$ in $K_{n'}^c$.
Finally, by the `absorbing' property of $C$, $G$ contains a properly coloured cycle $C'$ with $V(C') = V(P) \cup V(C) = V(K_n^c)$.
Therefore, $C'$ is a properly coloured Hamiltonian cycle as required.

The paper is organised as follows.
In the next section, we set up some basic notation and give some extremal examples to show that Conjecture~\ref{conj} is sharp.
Section~\ref{sec:abscycle} and Section~\ref{sec:2fact} are devoted to finding a small absorbing cycle and a properly coloured $2$-factor respectively.
Finally, we prove Theorem~\ref{PCHC} in Section~\ref{sec:proof}.

%%%%%%%%%%%%%%%%%%%%%%%%%%%%%%%%%%%%%%%%%%%%%%%%%%%%%%%%%%%%%%%%%%%%%%%%%%%%%%%%%%%%%%%%%%%%%%%%%%%%%%%%%%%%%%

\section{Notation and extremal examples} \label{notation}

Throughout this paper, unless stated otherwise, $c$ is assumed to be an edge-colouring. 
Hence, $c(xy)$ is the colour of the edge~$xy$.
For $v \in V(G)$, we denote by $N_G(v)$ the neighbourhood of~$v$ in $G$.
If the graph $G$ is clear from the context, we omit the subscript.

Given a vertex set $U\subseteq V(G)$, write $G[U]$ for the (edge-coloured) subgraph of $G$ induced by~$U$.
We write $G \setminus U$ for the graph obtained from $G$ by deleting all vertices in~$U$.
For a vertex $u$, we sometime write $u$ to mean the set $\{u\}$.
Given a subgraph $H$ in $G$, we write $G -H$ for the graph obtained from $G$ by deleting all edges in~$H$.
For edge-disjoint graphs $G$ and $H'$, we denote by $G + H'$ the union of $G$ and $H'$.
We write $G - H +H'$ to mean $(G - H ) + H'$.

Let $U,W \subseteq V(G)$ not necessarily disjoint.
Whenever we define an auxiliary bipartite graph $H$ with vertex classes $U$ and~$W$, we mean that $H$ has vertex classes $U'$ and $W'$, where $U'$ is a copy of~$U$ and $W'$ is a copy of~$W$.
Hence, $U$ and $W$ are considered to be disjoint in~$H$.
Given an edge $uw$ in~$H$, we say that $u \in U$ and $w \in W$ to mean that $u \in U'$ and $w \in W'$.

Every path $P$ is assumed to be directed.
Hence, the paths $v_1 v_2 \dots v_{\ell}$ and $v_{\ell} v_{\ell-1} \dots v_1$ are considered different for $\ell \ge 2$.
Note that $|P|$ denotes the order of $P$. 
Given a path $P= v_1 v_2 \dots v_{\ell}$ and a vertex $x \in N(v_1) \setminus V(P)$, we define $xP$ to be the path $x  v_1 v_2 \dots v_{\ell}$. 
Similarly, given vertex-disjoint paths $P_1, \dots, P_s$, we define the path $P_1 \dots P_s$ to be the concatenation of $P_1, \dots, P_s$ (if it exists).

%%%%%%%%%%%%%%%%%%%%%%%%%%%%%%%%%%%%%%%%%%%%%%%%%%%%%%%%%%%%%%%%%%%%%%%%%%%%%%%%%%%%

\subsection{Extremal examples} \label{sec:example}

We now present some edge-colourings on $K_n$ to show that Conjecture~\ref{conj} is sharp.
The first example was given by Bollob\'as and Erd\H{o}s~\cite{MR0411999} for $n \equiv  1 \pmod{4}$.

\begin{exm}
Consider $n = 4k+1$.
Let $G$ be a $2k$-regular graph on $n$ vertices. 
Note that the compliment $\overline{G}$ of $G$ is also a $2k$-regular graph.
Let $K_n^c$ be obtained by colouring all edges of $G$ red and all edges of $\overline{G}$ blue.
Note that $\Dmon(K_n^c) = 2k = \lfloor n/2 \rfloor$.
However, $K_n^c$ does not contain any properly coloured Hamiltonian cycle $C_n$, since the edge-chromatic number of $C_n$ is~$3$.
\end{exm}

For even~$n$, Fujita and Magnant~\cite{FujitaMagnant09} showed that there exists a $K_{n}^c$ with $\delta^c(K_{n}^c) = n/2$ with no properly coloured Hamiltonian cycle.
In fact, their example also satisfies $\Dmon(K_{n}^c) = n/2$.
Hence, Conjecture~\ref{conj} is also sharp for even~$n$.
The example given by Fujita and Magnant is derived from a tournament on $n$ vertices, that is, an oriented complete graph.
In the proposition below, we present a simple generalization of their construction for general oriented graphs.
Given an oriented graph $\overrightarrow{G}$, let $d^-_{\overrightarrow{G}}(v)$ and $d^+_{\overrightarrow{G}}(v)$ be the in- and outdegree of a vertex $v \in V(\overrightarrow{G})$.
Also, define the \emph{maximum indegree $\Delta^- (\overrightarrow{G})$ of $\overrightarrow{G}$} to be the maximum $d^-_{\overrightarrow{G}} (v)$ over all vertices $v \in V(\overrightarrow{G})$.

\begin{prp} \label{construction}
Let $G$ be a graph.
Suppose that $\overrightarrow{G}$ is an oriented graph obtained by orienting each edge of~$G$.
Then there exists an edge-coloured graph $G^c$ obtained by colouring each edge of $G$ such that 
\begin{itemize}
	\item[\rm (i)] $\Dmon(G^c) = \Delta^-(\overrightarrow{G})$;
	\item[\rm (ii)] $d^c_{G^c}(v) = d^+_{\overrightarrow{G}} (v) + \min \left\{ 1, d^-_{\overrightarrow{G}} (v) \right\}$ for all $v \in V(G)$;
	\item[\rm (iii)] $C$ is a properly coloured cycle in~$G^c$ if and only if $C$ is a directed cycle in~$\overrightarrow{G}$.
\end{itemize}
\end{prp}

\begin{proof}
Let $\{c_x: x \in V(G)\}$ be a set of distinct colours.
Define an edge-colouring $c$ of $G$ such that for every edge $xy \in E(G)$, $c(xy) = c_y$ if and only if $\overrightarrow{xy}$ is in $\overrightarrow{G}$.
Let $G^c$ be the graph $G$ with edge-colouring~$c$.
The proposition follows.
\end{proof}

Let $T_{2m}$ be a tournament on $2m$ vertices obtained from a regular tournament $T$ on $2m-1$ vertices by adding a directed edge from a new vertex~$x$ to every $y \in V(T)$.
Note that $\Delta^-(T_{2m}) = m$ and $T_{2m}$ does not contain any directed Hamiltonian cycle.
Therefore, by Proposition~\ref{construction}, there exists a $K_{2m}^c$ (corresponding to $T_{2m}$) with $\Delta(K_{2m}^c) = \Delta^-(T_{2m}) = m$ and $\delta^c(K_{2m}^c) = m$ that does not contain any properly coloured Hamiltonian cycle.

In the proposition below, we present yet another $K_n^c$, which also shows that Conjecture~\ref{conj} is sharp for even~$n$.
Moreover, this construction of $K_n^c$ can be generalized to forbid any properly coloured paths and cycles of arbitrary length.
We would like to point out that, by a suitable choice of tournament, Proposition~\ref{construction} also yields the same result for properly coloured cycles but not for properly coloured paths.
An edge-coloured graph $G$ is \emph{rainbow} if all edges have distinct colours.

\begin{prp} \label{prp:excol}
Let $\ell$ and $n$ be integers with $1 \le \ell \le n/2$.
Then there exists an edge-coloured graph $K_{n}^c$ on $n$ vertices with $\Dmon(K_{n}^c) =  n - \ell $ and $\delta^c(K_n^c)= \ell$ such that all properly coloured cycles in $K_{n}^c$ have length less than $2\ell$ and all properly coloured paths in $K_{n}^c$ have length less than $2\ell+1$.
\end{prp}

\begin{proof}
Let the vertices of $K_{n}$ be $x_1, \dots, x_{\ell}, y_1, \dots,y_{n - \ell}$.
Set $X = \{x_1, \dots, x_{\ell}\}$ and $Y = \{ y_1, \dots,y_{n - \ell} \}$.
Let $c : E(K_n) \rightarrow \mathbb{N}$ be an edge-colouring of $K_n$ such that:
\begin{itemize}
 \item[(a)] $c(x_iy_j) = i$ for all $1 \le i \le \ell$ and all $1 \le j \le n - \ell$;
 \item[(b)] $c(y_iy_j) = 1$ for all $1 \le i < j \le n - \ell$;
 \item[(c)] $K_n[X]$ is rainbow under $c$ and does not contain any colour in $\{ 1, \dots, {\ell} \}$.
\end{itemize}
Note that $\Dmon(K_n^c) = |Y| = n -\ell$ and $\delta^c(K_n^c) = |X| = \ell$.

Suppose $C$ is a properly coloured cycle in~$K_n^c$.
Let $P_1, P_2, \dots,P_r$ be the paths of $C$ induced on the vertex set $Y$ (where $P_i$ may consist of one vertex).
Since each $P_i$ is properly coloured, (b) implies that $1 \le |P_i| \le 2 $.
Note that after seeing one $P_i$ we must immediately see at least two consecutive vertices in $X$, so 
\begin{align}
\label{XY}
|X| \ge |X \cap V(C)| \ge 2 r  \ge 2 \lceil |Y \cap V(C)| /2 \rceil \ge |Y \cap V(C)|.
\end{align}
Therefore 
\begin{align*}
|C| = |X \cap V(C)| + |Y \cap V(C) | \le 2 |X| = 2 \ell.
\end{align*}
If $|C| = 2 \ell$, then we must have equality in~\eqref{XY} and so $|X| = 2r = |Y \cap V(C)|$.
Hence, we must have $|P_i| =2$ for all $i \le r$.
Thus, each $P_i$ is an edge of colour~$1$ by~(b).
Therefore, after seeing one $P_i$ we must see at least two vertices $x_j$ with $j \ne 1$ by (a) and~(b) before seeing another $P_{i'}$.
This implies that $|X \setminus  x_1  | \ge 2r = |X|$, a contradiction. 
Hence, all properly coloured cycles in $K_n^c$ have length less than $2 \ell$.
A similar argument shows that all properly coloured paths in $K_n^c$ have length less than $2 \ell+1$.
\end{proof}

%%%%%%%%%%%%%%%%%%%%%%%%%%%%%%%%%%%%%%%%%%%%%%%%%%%%%%%%%%%%%%%%%%%%%%%%%%%%%%%%%%%%%%%%%%%%%%%%%%%%%%%%%%%%%%
\section{absorbing cycle} \label{sec:abscycle}

% Let $\eps >0$ and let $n$ be a large integer.
% Let $K_n^c$ be an edge-coloured $K_n$ with $\Dmon ( K_n^c ) \le (1/2-\eps)n$.
The aim of this section is to show that there exists a small cycle $C$ in $K_n^c$ such that, for any properly coloured path $P$ with $V(C) \cap V(P) = \emptyset$ and $|P| \ge 4$, there exists a properly coloured cycle $C'$ with $V(C') = V(P) \cup V(C)$.

\begin{lma}[Absorbing cycle lemma] \label{lma:abscycle}
Let $0 < \eps < 1/2$.
There exists an integer $n_0$ such that the following holds whenever $n \ge n_0$.
Suppose that $K_n^c$ is an edge-coloured $K_n$ with $\Dmon(K_n^c) \le (1/2 - \eps)n $.
Then there exists a properly coloured cycle $C$ with $|C| \le 2^{-5} \eps^{4 \eps^{-2}+2} n$ such that, 
for any properly coloured path $P$ in $K_n^c \setminus V(C)$ with $|P| \ge 4$, $K_n^c$ contains a properly coloured cycle $C'$ with $V(C') = V(C) \cup V(P)$.
\end{lma}

We will need the following definition.

\begin{dfn}
Let $x_1$, $x_2$, $y_1$, $y_2$ be distinct vertices in $V(K_n^c)$.
A path $P$ is an \emph{absorbing path for $(x_1, x_2; y_1, y_2)$} if the following conditions hold:
\begin{enumerate}
	\item[(i)] $P = z_1z_2z_3z_4$ is a properly coloured path of order~$4$; 
	\item[(ii)] $V(P) \cap  \{ x_1, x_2, y_1, y_2\} = \emptyset$;
	\item[(iii)] both $z_1 z_2 x_1 x_2 $ and $y_1 y_2 z_3 z_4$ are properly coloured paths.
\end{enumerate}
\end{dfn}
Note that the ordering of $(x_1, x_2; y_1, y_2)$ is important.
Given distinct vertices $x_1,x_2 ,y_1, y_2$, let $\mathcal{L}(x_1, x_2; y_1, y_2)$ be the set of absorbing paths $P$ for $(x_1, x_2; y_1, y_2)$.
By the definition of an absorbing path, we have the following proposition.

\begin{prp} \label{prp:abspath}
Let $P' = x_1 x_2 \dots x_{\ell-1} x_{\ell}$ be a properly coloured path with $\ell \ge 4$. 
Let $P = z_1 z_2 z_3 z_4$ be an absorbing path for $(x_1, x_2; x_{\ell-1}, x_{\ell} )$ with $V(P) \cap V(P') = \emptyset$.
Then $z_1z_2 P' z_3 z_4 = z_1 z_2 x_1 x_2 \dots x_{\ell -1} x_{\ell} z_3z_4$ is a properly coloured path.
\end{prp}

Lemma~\ref{lma:abscycle} will be proved as follows.
Suppose that $\Dmon(K_n^c) \le ( 1/2 - \eps ) n$. 
In the next lemma, Lemma~\ref{lma:abspath}, we show that $\mathcal{L}(x_1, x_2; y_1, y_2)$ is large for any distinct $x_1,x_2 ,y_1, y_2 \in V(K_n^c)$.
By a simple probabilistic argument, Lemma~\ref{lma:absorbing} shows that there exists a small family $\mathcal{F}'$ of vertex-disjoint properly coloured paths (of order~$4$) such that, for any distinct $x_1,x_2 ,y_1, y_2 \in V(K_n^c)$, $\mathcal{F}'$ contains at least one absorbing path for $(x_1, x_2; y_1, y_2)$.
Finally, we join all paths in $\mathcal{F}'$ into one short properly coloured cycle $C$ using Lemma~\ref{lma:ifar}.
Moreover, $C$ satisfies the desired property in Lemma~\ref{lma:abscycle}.

\begin{lma} \label{lma:abspath}
Let $0 < \eps < 1/8$ and let $n \ge 5 \eps^{-1}$ be an integer.
Suppose that $K_n^c$ is an edge-coloured $K_n$ with $\Dmon(K_n^c) \le (1/2 - \eps)n $.
Then $| \mathcal{L}(x_1, x_2; y_1, y_2) |\ge \eps^2 n^4/4 $ for all distinct vertices $x_1,x_2 ,y_1, y_2 \in V(K_n^c)$.
\end{lma}

\begin{proof}
Fix distinct vertices $x_1,x_2 ,y_1, y_2 \in V(K_n^c)$.
Set $V' = V(K_n^c) \setminus \{x_1, x_2, y_1, y_2\}$ and $\Delta = \Dmon(K_n^c)$.
We can find two distinct vertices $z_1$, $z_2$ in $V'$ such that $z_1 z_2x_1x_2$ is a properly coloured path.
Note that there are $(|V'| - (\Delta-1)) (|V'| - (\Delta-1)-1) \ge n^2/4$ choices for $z_1$ and $z_2$.
The number of vertices $z_3 \in V' \setminus \{z_1,z_2\}$ such that $c(z_3z_2) \ne c(z_2 z_1)$ and $c(z_3y_2 ) \ne c(y_1y_2)$ is at least
\begin{align*}
	 (|V'|-2) - 2( \Delta -1) = |V'| - 2\Delta \ge 2 \eps n - 4 \ge \eps n.
\end{align*}
Pick one such $z_3$.
By a similar argument, the number of vertices $z_4 \in V' \setminus \{z_1,z_2,z_3\}$ such that $c(z_3 z_4) \ne c(z_3 y_2)$ and $c(z_3 z_4) \ne c(z_3 z_2 ) $ is at least $\eps n $. 
Pick one such $z_4$.
Notice that $z_1z_2z_3z_4$ is an absorbing path for $(x_1, x_2; y_1, y_2)$.
Furthermore, there are at least $n^2/4 \times \eps n \times \eps n = \eps^2 n^4/4$ many choices of $z_1$, $z_2$, $z_3$ and~$z_4$.
Therefore, the proof is completed.
\end{proof}

The next lemma is proved by a simple probabilistic argument since each $\mathcal{L}(x_1, x_2; y_1, y_2)$ is large.
We will need the following Chernoff bound for the binomial distribution
(see e.g.~\cite{MR1885388}).
Recall that the binomial random variable with parameters $(n,p)$ is the sum
of $n$ independent Bernoulli variables, each taking value $1$ with probability $p$,
or $0$ with probability $1-p$.

\begin{prp}\label{prop:chernoff}
Suppose that $X$ has the binomial distribution and $0<a<3/2$.
Then $\mathbb{P}(|X - \mathbb{E}X| \ge a\mathbb{E}X) \le 2 e^{-a^2\mathbb{E}X /3}$.
\end{prp}

\begin{lma} \label{lma:absorbing}
Let $0 < \eps < 1/2$.
Then there exists an integer $n_0$ such that whenever $n \ge n_0$ the following holds.
Suppose that $K_n^c$ is an edge-coloured $K_n$ with $\Dmon(K_n^c) \le (1/2 - \eps)n $.
Then there exists a family $\mathcal{F}'$ of vertex-disjoint properly coloured paths of order~$4$ such that $|\mathcal{F}'| \le 2^{-7} \eps^2 n$ and, for all distinct vertices $x_1, x_2, y_1, y_2 \in V(K_n^c)$, $|\mathcal{L}(x_1, x_2; y_1, y_2) \cap \mathcal{F}'| \ge 1$.
\end{lma}

\begin{proof}
Fix $0 < \eps <1/2$ and let $n_0 $ be a sufficiently large integer.
Let $K_n^c$ be an edge-coloured $K_n$ with $\Dmon(K_n^c) \le (1/2 - \eps)n $ and $n \ge n_0$.
Recall that each path is assumed to be directed.
A path $z_1 z_2 z_3z_4$ will be considered as a $4$-tuple $(z_1, z_2, z_3, z_4)$.
Choose a family~$\mathcal{F}$ of $4$-tuples in $V(K_n^c)$ by selecting each of the $n!/ (n-4)!$ possible $4$-tuples independently at random with probability 
\begin{align*}
p = 2^{-8} \eps^{2} \frac{ (n-4 )!}{(n-1)!} >  2^{-8}\eps^2 n^{-3}.
\end{align*}
Then by Proposition~\ref{prop:chernoff}
\begin{align}
|\mathcal{F}| & \le 2 \mathbb{E}|\mathcal{F}| =  2 p \frac{n!}{(n-4)!} = 2^{-7} \eps^{2}n \label{eqn:|F|}
\end{align}
with probability at least  $1 - 2e^{- \mathbb{E}|\mathcal{F}|/3} = 1 - 2e^{-  \eps^2 n/( 3 \times 2^8) } \ge 5/6$ since $n$ is large.
By Lemma~\ref{lma:abspath}, for every distinct vertices $x_1,x_2,y_1,y_2$, we have $|\mathcal{L} (x_1, x_2; y_1, y_2)| \ge \eps^2 n^4/4$.
Hence, given distinct vertices $x_1,x_2,y_1,y_2$, by Proposition~\ref{prop:chernoff} we have 
\begin{align}
|\mathcal{L}(x_1, x_2; y_1, y_2) \cap \mathcal{F}| & \ge  \mathbb{E}  |\mathcal{L}(x_1, x_2; y_1, y_2)  \cap \mathcal{F} | / 2 \nonumber \\
& = p |\mathcal{L}(x_1, x_2; y_1, y_2)| /2
> 2^{-11} \eps^{4} n
\label{eqn:L}
\end{align} 
with probability at least $1 - 2 e^{ -\mathbb{E} |\mathcal{L}(x_1, x_2; y_1, y_2)  \cap \mathcal{F} |/12 } > 1 -  2 e^{- c_0 n}$, where $c_0 = \eps^4 /(3 \times 2^{12})$.
By the union bound, $\mathcal{F}$ satisfies~\eqref{eqn:L} for all distinct vertices $x_1,x_2,y_1,y_2$ with probability at least $1 - 2 n^4 e^{- c_0 n} \ge 5/6$ since $n$ is large.

We say that two $4$-tuples $(a_1,a_2,a_3,a_4)$ and $(b_1,b_2,b_3,b_4)$ are \emph{intersecting} if $a_i = b_j$ for some $1 \le i,j\le 4$.
Furthermore, we can bound the expected number of intersecting pairs of $4$-tuples in $\mathcal{F}$ from above by
\begin{align}
	\frac{n!}{(n-4)!} \times 4^2 \times \frac{(n-1)!}{(n-4)!} \times p^2 = 2^{-12}\eps^4 n.
\nonumber
\end{align}
Thus, using Markov's inequality, we derive that with probability at least~$1/2$
\begin{align}
	\text{$\mathcal{F}$ contains at most $ 2^{-11}\eps^4 n$ intersecting pairs of $4$-tuples.} \label{eqn:F}
\end{align}
Hence, with positive probability, the family $\mathcal{F}$ satisfies \eqref{eqn:|F|}, \eqref{eqn:L} for all distinct vertices $x_1,x_2,y_1,y_2$, and~\eqref{eqn:F}.
Pick one such~$\mathcal{F}$.
We delete one $4$-tuple in each intersecting pair in~$\mathcal{F}$.
We further remove those $4$-tuples that are not absorbing paths.
We call the resulting family~$\mathcal{F}'$.
Note that $\mathcal{F}'$ satisfies
\begin{align}
|\mathcal{L}(x_1, x_2; y_1 ,y_2) \cap \mathcal{F}'| > & 2^{-11}\eps^4 n - 2^{-11}\eps^4 n = 0\nonumber
\end{align}
for all distinct vertices $x_1,x_2, y_1, y_2 \in V(K_n^c)$.
Since $\mathcal{F}'$ consists of pairwise disjoint $4$-tuples and each $4$-tuple in $\mathcal{F}'$ is an absorbing path, $\mathcal{F}'$ is a set of vertex-disjoint properly coloured paths of order~$4$.
\end{proof}

As mentioned earlier, in order to prove Lemma~\ref{lma:abscycle}, we join the paths in $\mathcal{F}'$ given by Lemma~\ref{lma:absorbing} into a short properly coloured cycle.
The lemma below shows that we join any two disjoint edges by a properly coloured path of constant length.

\begin{lma} \label{lma:ifar}
Let $0 < \eps < 1/100$.
Then there exists an integer $n_0$ such that whenever $n \ge n_0$ the following holds.
Suppose that $K_n^c$ is edge-coloured with $\Dmon(K_n^c) \le (1/2 - \eps)n $.
Let $v_1$, $v_2$, $v'_1$, $v'_2$ be distinct vertices.
Then there exists an integer $2 \le i_0 \le 2 \eps^{-2}$  such that there are at least $(\eps^2 n)^{i_0}$ paths $P$ with $|P| = i_0$ and $v_1v_2 P v'_1v'_2$ is a properly coloured path. 
\end{lma}

To illustrate the idea of the proof, we consider the following simpler problem.
Suppose that $K_n^c$ with $\Dmon(K_n^c) \le n/3-3$ and let $x_1,x_2, y \in V(K_n^c)$ be distinct.
We claim that there exist distinct vertices $w$ and $v$ such that $x_1x_2 w v y$ is a properly coloured path.
(In other words, we can join an edge $x_1x_2$ and a vertex $y$ into a properly coloured path of order~$5$.)
Let $V' = V(K_n^c) \setminus \{x_1,x_2,y\}$ and let $W$ be the set of vertices $w \in V'$ such that $c(x_1x_2) \ne c(x_2 w)$.
Hence, $|W| = n - 1 - \Dmon(K_n^c) - 1 > 2n/3$.
Define an auxiliary bipartite graph $H$ with vertex classes $W$ and $V'$ and edge set $E(H)$ such that for $w \in W$ and $v \in V'$, $w v \in E(H)$ if and only if $c(wv) \ne c(w x_2)$ in~$K_n^c$.
(Recall Section~\ref{notation} that we consider $W$ and $V'$ to be disjoint in $H$.)
Hence, if $w v \in E(H)$, then $x_1 x_2 w v$ is a properly coloured path.
Also, every $w \in W$ has degree at least $2n/3$ in~$H$. 
By an averaging argument, there exists a vertex $v \in V'$ with degree at least $4n/9$ in~$H$.
Recall that $\Dmon(K_n^c) \le n/3-3 < 4n /9$.
There exist distinct $w,w' \in N_H(v)$ such that $c(w v) \ne c(w'v )$.
Therefore, $x_1x_2 w v y$ or $x_1x_2 w' v y$ is a properly coloured path as claimed.

\begin{proof}[Proof of Lemma~$\ref{lma:ifar}$]
Fix $\eps$ and let $n_0$ be a sufficiently large integer.
Let $K_n^c$ be an edge-coloured complete graph with $\Dmon(K_n^c) \le (1/2 - \eps) n$ and $n \ge n_0$.
Set $\Delta = \Dmon(K_n^c)$, $V' = V(K_n^c) \setminus \{v_1,v_2, v'_1, v'_2\}$ and $n' = |V'| =  n - 4$.
We omit floors and ceilings for clarity of presentation.

For integers $i \ge 0$, we say that a vertex $x \in V'$ is \emph{$i$-far} from $(v_1,v_2)$ if there exist at least $(\eps^2 n)^i$ paths $P$ with $V(P) \subseteq V' \setminus x$ and $|P| = i$ such that $v_1v_2 P x$ is a properly coloured path. 
Note that any vertex $x \in V'$ with $c(xv_2) \ne c(v_1v_2)$ is $0$-far.
A vertex $x$ is \emph{strongly $i$-far} if for any colour $c'$, after removing all edges $xy$ with $c(xy)= c'$ there still exist at least $(\eps^2 n)^i/2$ paths $P$ with $V(P) \subseteq V' \setminus x$ and $|P| = i$ such that $v_1v_2 P x$ is a properly coloured path. 
Hence, if $x$ is $i$-far but not strongly $i$-far, then there exists a unique colour $c_i(x)$ such that $x$ is no longer $i$-far after removing all edges~$xy$ with $c(xy) = c_i(x)$.
Moreover, there are at least $(\eps^2 n)^i/2$ paths $P$ with $V(P) \subseteq V' \setminus x$ and $|P| = i$ such that $v_1v_2 P x$ is a properly coloured path and the edge (in $P$) incident with $x$ is of colour $c_i(x)$.
Note that no vertex is strongly $0$-far.

For integers $i \ge 0 $, let $X_i$ be the set of vertices in $V'$ that are $i$-far but not strongly $i$-far.
Also, let $Y_i$ be the set of vertices in $V'$ that are strongly $i$-far.
Note that $Y_0 = \emptyset$.
Let $N' = \{ w \in V' : c(w v'_1) \ne c(v_1' v_2') \}$.
If $y \in N' \cap Y_i$, then there exist at least $(\eps^2 n )^i/2 $ paths $P$
such that $V(P) \subseteq V' \setminus y$, $|P| = i$ and moreover $v_1v_2P y v'_1 v'_2$ is a properly coloured path.
Hence, if $|N' \cap Y_i | \ge 2 \eps^2 n$ for some $i \le  2 \eps^{-2}-1$, then the lemma holds by setting $i_0 = i+1$.
Recall that $\Dmon(K_n^c) = \Delta$, so $|N'| \ge n' - \Delta$.
Therefore, to prove the lemma, it is enough to show that $|Y_i| \ge \Delta + 2 \eps^2 n$ for some integer $1 \le i \le  2 \eps^{-2}-1$.

Recall that if $x \in X_i$, then there is a unique colour $c_i(x)$ such that $x$ is no longer $i$-far after removing all edges $xy$ with $c(xy) = c_i(x)$.
For each integer $0 \le i \le  2 \eps^{-2}-1$, define an auxiliary bipartite graph $H_i$ with vertex classes $X_i \cup Y_i$ and $V'$ and edge set $E(H_i)$ such that
\begin{itemize}
	\item[(a)] every vertex $y$ in $Y_i$ is adjacent to every vertex in $V' \setminus y$, and 
	\item[(b)] for $x \in X_i$ and $v \in V' \setminus x$, $xv$ is an edge in $H_i$ if and only if $c(xv) \ne c_i(x)$.
\end{itemize}
Since $\Dmon(K_n^c) = \Delta \le (1/2 - \eps)n$, each vertex $x \in X_i$ has degree at least $n' - 1 - \Delta \ge (2 + 3\eps)n/4$ in~$H_i$.
Thus, 
\begin{align}
	e(H_i) & \ge (2 + 3\eps) n |X_i| / 4 + (n'-1) |Y_i| . \label{eqn:eH_ilower}
\end{align}
Since $Y_0 = \emptyset$ and $X_0$ is the set of vertices $x \in V'$ such that $c(v_2x) \ne c(v_1v_2)$, we have $|X_0| \ge n' - \Delta \ge n/2$.
Thus, 
\begin{align}
e(H_0) \ge (2+ 3\eps)|X_0| n/4 \ge nn'/4.	\label{eqn:H'_0}
\end{align}
Suppose that $xv$ is an edge in $H_i$ with $x \in X_i$ and $v \in V'$.
Note that $v$ is in at most $i n^{i-1} \le (\eps^2 n )^{i}/4$ paths $P$ with $|P| = i$.
Since $x$ is $i$-far and $c(xv) \ne c_i(x)$, there exist at least $(\eps^2 n )^i /4$ paths $P$ such that $V(P) \subseteq V' \setminus v$, $|P| = i$ and $v_1v_2P x v$ is a properly coloured path.
A similar statement also holds for edges $yv$ in $H_i$ with $y \in Y_i$ and $v \in V'$.
Therefore, if a vertex $v \in V'$ has degree at least $4 \eps^2 n$ in $H_i$, then $v$ is $(i+1)$-far.
Similarly, we conclude that if $v$ has degree at least $\Delta + 4 \eps^2 n$ in $H_i$, then $v$ is strongly $(i+1)$-far. 
By counting the degrees of $v \in V'$ in $H_i$, we deduce that 
\begin{align}
	e(H_{i}) & \le 4 \eps^2 n |V' \setminus (X_{i+1} \cup Y_{i+1})| + ( \Delta + 4 \eps^2 n ) |X_{i+1}| + (|X_{i}| + |Y_{i}|) |Y_{i+1}| \nonumber \\
	& \le 4 \eps^2  n' n + \Delta  |X_{i+1}|+ (n'-1) |Y_{i+1}| \label{eqn:eH_iupper}\\
	& \le e(H_{i+1}) - \eps n ( 7 |X_{i+1}|/4 - 4 \eps n'), \nonumber
\end{align}
where the last inequality is due to~\eqref{eqn:eH_ilower}.
Thus, 
\begin{align}
\text{if $|X_{i+1}| \ge 20 \eps n'/7$, then $e(H_{i+1}) \ge e(H_i) + (\eps n')^2$.}  \label{eqn:H_iupper3}
\end{align}
Since $e(H_i)$ is at most $n'^2$, there exists an integer $i' \le \eps^{-2}$ such that $|X_{i'+1}| < 20 \eps n'/7$.
Let $i'$ be the smallest integer such that $|X_{i'+1}| < 20 \eps n'/7$.
Hence, $e(H_{i}) \ge e(H_0)  \ge n n'/4$ by \eqref{eqn:H_iupper3} and~\eqref{eqn:H'_0}.
By~\eqref{eqn:eH_iupper},
\begin{align*}
	n n'/4 &  \le e(H_{i'})  \le 4 \eps^2  n' n + \Delta  |X_{i'+1}| + n'  |Y_{i'+1} | \\
	&  \le 2 \eps n' n + n' |Y_{i'+1}|	.
\end{align*}
Hence, $|Y_{i'+1}|  \ge (1/4 - 2 \eps) n$.
Therefore, in $H_{i'+1}$, each vertex in $V'$ has degree at least $ |Y_{i'+1}| -1 \ge (1/4 - 2 \eps) n-1 \ge 4 \eps^2 n$.
This implies that $X_{i'+2}  \cup Y_{i'+2} = V'$ and so 
\begin{align}
e(H_{i'+2}) \ge (2+ 3\eps) n' n/4 \label{eqn:Hi'}
\end{align}
by~\eqref{eqn:eH_ilower}.
Let $i''$ be the smallest integer $i \ge i'+2$ such that $|X_{i+1}| < 20 \eps n'/7$.
Since $e(H_i) \le n'^2$, $i''$ exists by \eqref{eqn:H_iupper3}.
Moreover, $i'' \le i'+2  + 1/(2\eps^2) < 2\eps^{-2}-3$ as $e(H_{i'+2})  \ge (n')^2/2$.
Note that $e(H_{i''}) \ge e(H_{i'+2}) \ge (2+ 3\eps) n' n/4$ by~\eqref{eqn:Hi'}.
By~\eqref{eqn:eH_iupper}, we have
\begin{align*}
	(2+ 3\eps) n' n/4 & \le e(H_{i''})  \le 4 \eps^2  n' n + \Delta  |X_{i''+1}|  + (n'-1) |Y_{i''+1}|  \\
	 & \le  ( 4 \eps^2 + 10\eps/7   )  n' n + n'|Y_{i''+1}|, \\
	|Y_{i''+1}| & \ge ( 1/2 -  \eps + 2 \eps^2) n \ge \Delta + 2 \eps^2 n.
\end{align*}
This completes the proof of the lemma.
\end{proof}

We are ready to prove Lemma~\ref{lma:abscycle}.

\begin{proof}[Proof of Lemma~$\ref{lma:abscycle}$]
Since $\eps^{4 \eps^{-2}+2}$ is an increasing function of $\eps$ (as $0<\eps <1/2$), it suffices to prove the lemma for $\eps < 1/100$.
Let $n_0$ be a sufficiently large integer. 
Let $K_n^c$ be an edge-coloured complete graph with $\Dmon(K_n^c) \le (1/2 - \eps) n$ with $n \ge n_0$.
Set $\gamma =  \eps^{2 \eps^{-2}+2}$.
Let $\mathcal{F}'$ be the set of properly coloured paths obtained by Lemma~\ref{lma:absorbing}.
Therefore,
\begin{align}
|\mathcal{F}'| & \le  2^{-7} \gamma^2 n =  2^{-7} \eps^{4 \eps^{-2}+4} n,  \label{eqn:F'1}\\
|\mathcal{L}(x_1, x_2; y_1, y_2) \cap \mathcal{F}'| & \ge 1 \nonumber
\end{align}
for all distinct vertices $x_1,x_2, y_1,y_2 \in V(K_n^c)$.

We join the paths in $\mathcal{F}'$ into a properly coloured cycle $C$ as follows.
Let $P_1, \dots, P_{|\mathcal{F}'|}$ be the properly coloured paths in $\mathcal{F}'$.
For each $1 \le j \le |\mathcal{F}'|$, we are going to find a path $Q_j$ with $|Q_j| \le 2 \eps^{-2}$ and $V(Q_j) \subseteq V(K_n^c) \setminus V( \bigcup \mathcal{F}')$ such that $P_j Q_j P_{j+1}$ is a properly coloured path, where we take $P_{|\mathcal{F}'|+1} = P_1$, and such that $V(Q_j) \cap V(Q_{j'}) = \emptyset$ for all $j \ne j'$.
Assume that we have already constructed $Q_1, \dots,Q_{j-1}$.
Let $P_j = v_1 v_2 v_3 v_{4}$ and $P_{j+1} = v'_1 v'_2 v'_3 v'_{4}$.
By Lemma~\ref{lma:ifar} (with $v_1 = v_{3}$ and $v_2 = v_{4}$), there exists an integer $2 \le i_0 \le 2 \eps^{-2} $ such that there are at least $(\eps^2 n)^{i_0}$ paths $Q$ with $|Q| = i_0$ such that $ v_{3} v_{4} Q v'_1 v'_2$ is a properly coloured path. 
Set $ \bigcup \mathcal{F}' = \bigcup_{F \in \mathcal{F}'} F $ and $W_j = V ( \bigcup \mathcal{F}') \cup \bigcup_{j' < j} V(Q_{j'})$, so  by~\eqref{eqn:F'1}
\begin{align*}
|W_j| & = | V (\bigcup \mathcal{F}')| + \sum_{j' < j} |Q_{j'}|
\le  4 |\mathcal{F}'|  + 2\eps^{-2}  (j-1) \\
%& \le 4 |\mathcal{F}'| + 2\eps^{-2}  ( |\mathcal{F}'| -1)
& <  (4+2 \eps^{-2} )|\mathcal{F}'| \le 2^{-5} \eps^{4 \eps^{-2}+2} n.
\end{align*}
Moreover, $W_j$ intersects with at most 
\begin{align*}
|W_j| \times i_0 n^{i_0-1} 
< 2^{-5} \eps^{4 \eps^{-2}+2} n \times 2 \eps^{-2} n^{i_0-1}
= 2^{-4 }\eps^{4 \eps^{-2}}  n^{i_0}
< (\eps^2 n)^{i_0}
\end{align*}
 paths of order~$i_0$ as $i_0 \le 2\eps^{-2}$. 
Therefore, there exists a path $Q_j$ with $V(Q_j) \subseteq V(K_n^c) \setminus W_j$ such that $ v_{3}v_{4} Q_j v'_1v'_2$ is a properly coloured path, which implies that $P_jQ_jP_{j+1}$ is a properly coloured path.
Hence, we find properly coloured paths $Q_1, \dots, Q_{|\mathcal{F}'|}$ as desired.
This means that $K_n^c$ contains a properly coloured cycle $C$ obtained by concatenating $P_1, Q_1, P_2, Q_2, \dots , Q_{|\mathcal{F}'|}$.
Note that $|C| \le ( 4+ 2 \eps^{-2} )  | \mathcal{F}' | \le 2^{-5} \eps^{4 \eps^{-2}+2} n $ by~\eqref{eqn:F'1}.

We now show that $C$ has the desired `absorbing' property.
Let $P = x_1x_2  \dots x_{\ell}$ be a properly coloured path with $\ell \ge 4$ and $V(P) \cap V(C) = \emptyset$.
Pick $P' = z_1z_2z_3z_4 \in \mathcal{L}(x_1, x_2; x_{\ell-1}, x_{\ell}) \cap \mathcal{F}'$.
Since $P'$ is an absorbing path for $(x_1, x_2; x_{\ell-1}, x_{\ell})$, Proposition~\ref{prp:abspath} implies that $z_1z_2 P z_3z_4$ is a properly coloured path.
Note that the endedges are the same as in $P'$.
Therefore, there exists a properly coloured cycle $C'$ with $V(C') = V(C) \cup V(P)$.
This completes the proof of Lemma~\ref{lma:abscycle}.
\end{proof}

%%%%%%%%%%%%%%%%%%%%%%%%%%%%%%%%%%%%%%%%%%%%%%%%%%%%%%%%%%%%%%%%%%%%%%

\section{Properly coloured $2$-factors} \label{sec:2fact}

In this section, we prove the following lemma, which finds a properly coloured $2$-factor in $K_n^c$ with $\Dmon (K_n^c) < (1/2 - \eps) n $.

\begin{lma} \label{lma:2factor}
Let $0 < \eps < 1/2$.
Then there exists an integer $n_1 = n_1(\eps)$ such that every $K_n^c$ with $n \ge n_1$ and $\Dmon(K_n^c) \le (1/2 - \eps)n $ contains a properly coloured $2$-factor. 
\end{lma}

Before proving the lemma, we need the following notation.
Let $C$ be a directed cycle.
For a vertex $v \in V(C)$, let $v_+$ and $v_-$ be the successor and ancestor of $v$ in $C$ respectively. 
Further, let $c_-(v)$ and $c_+(v)$ be the colours $c(v v_-)$ and $c(v v_+)$ respectively.
For distinct vertices $u, v \in V(C)$, define $v C^+ u$ to be the path $v v_+ \dots u_- u$ on~$C$, and similarly define $v C^- u$ to be the path $v v_- \dots u_+ u$ on~$C$.

Given an edge-coloured graph $G$, we denote by $\mathcal{C}_G(v)$ the set of colours incident at $v$ in~$G$.
Equivalently, $\mathcal{C}_G(v) = \{ c(vu) : u \in N_G(v)\}$.
Given $x,y \in V(G)$, the \emph{distance $\dist_G(x,y)$ in $G$ between $x$ and $y$} is the minimum integer $\ell$ such that $G$ contains a path (not necessarily properly coloured) of length $\ell$ from $x$ to $y$. 
Note that $\dist_G (x,x) = 0$ for all $x \in V(G)$.
If $x$ and $y$ are not connected in~$G$, then we say $\dist_G (x,y) = \infty$.

Recall that a graph $G$ is said to be a $1$-path-cycle if $G$ is a vertex-disjoint union of at most one path $P$ and a number of cycles. 
We say that $G$ is a \emph{$1$-path-cycle with parameters $(x,c_x ; y, c_y)$} if $G$ satisfies the following three properties:
\begin{itemize}
	\item[(a)] $G$ is a properly coloured $1$-path-cycle;
	\item[(b)] the path $P = v_1 \dots v_{\ell}$ in $G$ has length at least 1 with $v_1 = x$ and $v_{\ell} = y$;
	\item[(c)] $c_x = c(v_1 v_2)$ and $c_y = c(v_{\ell} v_{\ell-1})$.
\end{itemize}
Note that $x$ and $y$ are the endvertices of $P$.
Also, $c_x$ and $c_y$ are precisely the colours of the edges in $P$ (and $G$) incident to $x$ and $y$ respectively.
The ordering of $(x,c_x ; y, c_y)$ is important.
Recall that all paths are assumed to be directed.
So `a $1$-path-cycle with parameters $(x,c_x ; y, c_y)$' is considered to be different from `a $1$-path-cycle with parameters $(y, c_y ; x,c_x)$', even though the underlying graphs maybe the same.
Let $G$ be a $1$-path-cycle with parameters $(x,c_x ; y, c_y)$ in $K_n^c$.
For a vertex $v \in V(G) \setminus x$, the edge $xv$ is a \emph{left chord for $G$} if $c(xv) \ne c_x$.
Similarly, the edge $yv$ is a \emph{right chord for $G$} if $v \in V(G) \setminus y $ and $c(yv) \ne c_y$.
A \emph{chord} is a left or right chord.

Now we sketch the proof of Lemma~\ref{lma:2factor}.
Suppose that $G$ is a properly coloured $1$-path-cycle in $K_n^c$ with $|G|$ maximal.
Further assume that $G$ has parameters $(x,c_x;y,c_y)$.
By chord rotations (defined later), we find a properly coloured $1$-path-cycle $G_0$ with parameters $(x',c_{x'};y',c_{y'})$ such that $c_{x'} \ne c(x'y') \ne c_{y'}$ and $V(G_0) = V(G)$.
So $x'y'$ is both a left and right chord for $G_{0}$.
Hence, $P_0 +x'y'$ is a properly coloured cycle, where $P_0$ is the path in $G_{0}$.
This implies that $G_{0} + x'y'$ is a set of properly coloured vertex-disjoint cycles. 
If $V(K_n^c) = V(G) = V(G_0)$, then $G_{0} + x'y'$ is a properly coloured $2$-factor .
If $V(K_n^c) \ne V(G)$, then $G_{0} + x'y'$ together with a vertex $z \in V(K_n^c) \setminus V(G)$ is a larger properly coloured $1$-path-cycle, contradicting the maximality of $|G|$.
This proves Lemma~\ref{lma:2factor}.

In the lemma below, we show why chords are useful.

\begin{lma} \label{rotation}
Let $G$ be a $1$-path-cycle with parameters $(x,c_x ; y, c_y)$.
Suppose that $yw$ is a right chord for $G$ with $w \notin \{ x \} \cup N_G(x) $.
Then there exists a properly coloured $1$-path-cycle $G'$ such that the following statements hold:
\begin{itemize}
\item[(i)] $G'$ is a spanning subgraph of $G + y w$ containing the edge $yw$.
\item[(ii)] $G'$ has parameters $(x,c_x ; y', c_{y'})$ such that $y' \in N_G(w)$,  $N_G(y') = \{ w, w' \}$ and $c_{y'}  = c(y'w') \in \mathcal{C}_G(y')$.
\item[(iii)] Let $N_G(w) = \{z_1,z_2\}$. 
Then $G'$ has parameters $(x,c_x; z_1, c' )$ only if $c(yw) \ne c(wz_2)$.
\item[(iv)] For $v \in V(G)$, if $\dist_G(v,z) \ge 2$ for all $z \in \{x,y,w\}$, then $ N_G(v) = N_{G'}(v)$.
\end{itemize}
Moreover, similar statements hold if $wx$ is a left chord with $w \notin \{ y \} \cup N_G(y)$.
\end{lma}

\begin{proof}
Let $yw$ be a right chord for $G$ and let $P= v_1 \dots v_{\ell}$ be the path in $G$, so $y = v_{\ell}$.
First suppose that $w \notin V(P)$, so $w \in V(C)$ for some properly coloured cycle $C$ in $G$.
Orient $C$ into a directed cycle so that $ c(yw) \ne c_-(w)$.
Observe that $P' = v_1 \dots v_{\ell} w C^- w_+$ is a properly coloured path.
Hence, $G' = G - C - P + P'$ is a properly coloured $1$-path-cycle containing $yw$ with parameters $(x,c_x ; w_+, c_+(w_+))$.
Hence (i) and (ii) hold for this case.
It is also easy to verify (iii) and~(iv).

Next, suppose that $w \in V(P)$ and so $w = v_j$ for some $3 \le  j \le \ell-2$.
Recall that $y = v_{\ell}$.
Note that $c(yw) = c(v_{\ell} v_j) \ne c(v_j v_{j-1})$ or $c(yw) = c(v_{\ell} v_j) \ne c(v_j v_{j+1})$.
If $c(v_{\ell} v_j) \ne c(v_j v_{j-1})$, then $P'' = v_1 \dots v_j v_{\ell} \dots v_{j+1}$ is a properly coloured path, see Figure~\ref{fig:C2}.
So $G ' = G - P + P''$ is a properly coloured $1$-path-cycle with parameters $(x,c_x ; v_{j+1}, c(v_{j+1} v_{j+2}))$.
If $c(v_{\ell} v_j) \ne c(v_j v_{j+1})$, then $C =  v_j \dots v_{\ell} v_{j}$ is a properly coloured cycle, see Figure~\ref{fig:C1}.
Hence, $G' = G - P + C + v_1\dots v_{j-1}$ is a properly coloured $1$-path-cycle with parameters $(x,c_x ; v_{j-1}, c(v_{j-1} v_{j-2}))$.
Hence (i)--(iv) follow.
\begin{figure}[tbp]
\begin{center}
\includegraphics[scale=0.6]{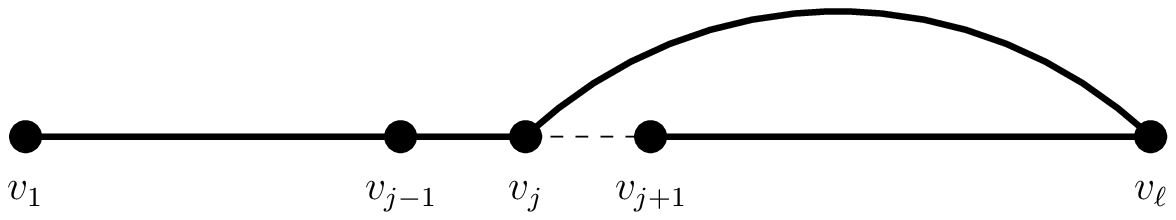}
\end{center}
\caption{$c(v_{\ell} v_j) \ne c(v_j v_{j-1})$}
\label{fig:C2}
\begin{center}
\includegraphics[scale=0.6]{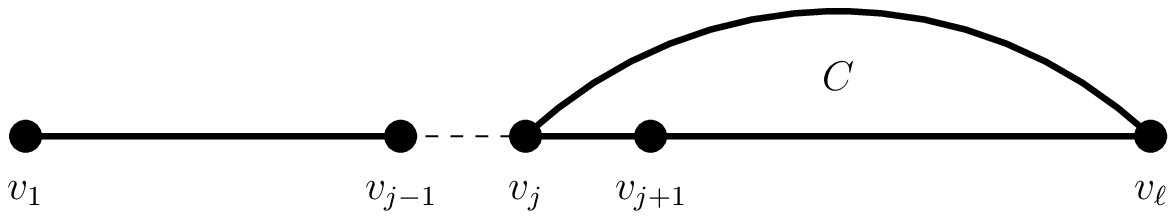}
\end{center}
\caption{$c(v_{\ell} v_j) \ne c(v_j v_{j+1})$}
\label{fig:C1}
\end{figure}
\end{proof}

Let $G$, $yw$ and $G'$ be as defined in Lemma~\ref{rotation}.
We say that $G'$ is obtained from $G$ by a \emph{chord rotation using the chord $yw$}, or a \emph{rotation using $yw$} for short.
Since this rotation changes the two parameters on the right and uses a right chord, we call this a \emph{right rotation}.
Similarly, we define a $\emph{left rotation}$ for a left chord~$xw$.

For the rest of this section, a chord $uw$ is either a left or right chord (but not both) unless stated otherwise.
Suppose that $G'$ is obtained from $G$ by a rotation using $uw$.
Since $|\{u,w\} \cap \{x,y\}|=1$, we can determine whether the chord (and rotation) is left or right by considering $\{u,w\}$.
Hence, we can write the chord $uw$ as an ordered pair $(u,w)$ with $u \in \{x,y\}$.
We simply write $uw$ for $(u,w)$ if the order is clear from the context.

Given a $1$-path-cycle $G$ with parameters $(x,c_x;y,c_y)$, we say that a $1$-path-cycle $G_{\ell}$ is obtained from $G$ by \emph{$\ell$ rotations using a chord sequence $e_1, \dots, e_{\ell}$} if there exist properly coloured $1$-path-cycles $G_1, \dots G_{\ell-1}$ such that for each $1 \le i \le \ell$ the following statements hold (by taking $G_0 = G$)
\begin{itemize}
\item[(a)] $e_i$ is a chord for $G_{i-1}$, and
\item[(b)] $G_i$ can be obtained from $G_{i-1}$ by a rotation using $e_i$.
\end{itemize}
The (chord) sequence $e_1, \dots, e_{\ell}$ with $e_i = ( u_i , w_i )$ is said to be \emph{spread out in~$G$} if the distance in~$G$ between any two elements in $\{x,y, w_i: 1 \le i \le \ell \}$ is greater than~$5$. 
Equivalently, $e_1, \dots, e_{\ell}$ is spread out in $G$ if $\dist_G(v,v') > 5$ for all distinct $v, v' \in \{x,y, w_i: 1 \le i \le \ell\}$.
The following corollary is proved by induction on $\ell$ together with Lemma~\ref{rotation}.

\begin{cor} \label{rotation2}
Let $G$ be a $1$-path-cycle with parameters $(x,c_x ; y, c_y)$.
Let $G_{\ell}$ be a $1$-path-cycle obtained from $G$ by $\ell$ rotations using a chord sequence $e_1, \dots, e_{\ell}$.
Suppose that the sequence $e_1, \dots, e_{\ell}$ is spread out in $G$ and $e_i = u_i w_i$ for all $1 \le i \le \ell$.
Then the following statements hold:
\begin{itemize}
	\item[(i)] $G_{\ell}$ has parameters $(x',c_{x'} ; y', c_{y'})$ with $V(G_{\ell}) = V(G)$, $x', y' \in \{x,y \} \cup \bigcup_{1 \le i \le \ell} N_G(w_i)$, $c_{x'} \in \mathcal{C}_G(x')$ and $c_{y'} \in \mathcal{C}_G(y')$.
	\item[(ii)] For $v \in V(G)$, if $\dist_G(v,u) > 5$ for all $u \in \{x,y, w_i : 1 \le i \le \ell\}$, then $v \notin N_{G_{\ell}}(x') \cup N_{G_{\ell}}(y')\cup \{x',y'\}$ and $ N_G(v) = N_{G_{\ell}}(v)$ 
	\item[(iii)] If $e_1, \dots, e_{\ell}$ are all right chords, then $x' = x$ and $y' \in N_G(w_{\ell})$.
	\item[(iv)] If $e_1, \dots, e_{\ell}$ are all left chords, then $x' \in N_G(w_{\ell})$ and $y' = y$.
\end{itemize}
\end{cor}

\begin{proof}
We proceed by induction on $\ell$.
The corollary is trivially true for $\ell =0$ and so we may assume that $\ell  \ge 1$.
Let $G_{\ell - 1}$ be the $1$-path-cycle with parameters $(x'',c_{x''} ; y'', c_{y''})$ obtained from $G$ by $\ell -1$ rotations using the chord sequence $e_1, \dots, e_{\ell-1}$.
Moreover, $G_{\ell}$ can be obtained from $G_{\ell-1}$ by a rotation using the chord $e_{\ell}$.
By the induction hypothesis, we have $x'', y'' \in \{x,y \} \cup \bigcup_{1 \le i \le \ell-1} N_G(w_i)$, $c_{x''} \in \mathcal{C}_G(x'')$ and $c_{y''} \in \mathcal{C}_G(y'')$.
Since $G_{\ell}$ can be obtained from  $G_{\ell-1}$ by using $e_{\ell}$, Lemma~\ref{rotation}(ii) implies (i) holds.
Similar arguments show that both (iii) and (iv) hold.

Let $v \in V(G)$ with $\dist_G(v,u) > 5$ for all $u \in \{x,y, w_i : 1 \le i \le \ell\}$.
Note that $N_G(v) = N_{G_{\ell-1}}(v)$ by the induction hypothesis.
Since $e_{\ell} = u_{\ell} w_{\ell}$ is a chord for $G_{\ell-1}$, $u_{\ell} \in \{x'', y''\} \subseteq \{x,y \} \cup \bigcup_{1 \le i \le \ell-1} N_G(w_i)$ by~(i).
Hence, $\dist_G(v,u) \ge 5$ for all $u \in \{x'',y'', w_{\ell}\}$.
By (i) and Lemma~\ref{rotation}(iv), we have $v \notin N_{G_{\ell}}(x') \cup N_{G_{\ell}}(y')\cup \{x',y'\}$ and $N_{G_{\ell}}(v) = N_{G_{\ell-1}}(v) = N_G(v)$.
Hence (ii) holds.
\end{proof}

The next lemma shows how to combine two chord sequences.

\begin{lma} \label{lma:rotation4}
Let $G$ be a properly coloured $1$-path-cycle with parameters $(x,c_x ; y, c_y)$.
Suppose that $G_R$ is a $1$-path-cycle with parameters $(x,c_x;z,c_z)$ obtained from $G$ by $\ell$ right rotations using a chord sequence $e_1, \dots,e_{\ell}$.
Suppose that $G_L$ is a $1$-path-cycle with parameters $(w,c_w;y,c_y)$ obtained from $G$ by $\ell'$ left rotations using a chord sequence $f_1, \dots,f_{\ell'}$.
Further suppose that the sequence $e_1, \dots,e_{\ell}, f_1, \dots,f_{\ell'}$ is spread out in $G$.
Then there exists a $1$-path-cycle $G_0$ with parameters $(w,c_w;z,c_z)$ obtained from $G$ by rotations using $e^1_1, \dots,e^1_{\ell}$, $f^1_1, \dots,f^1_{\ell'}$.
Moreover, $V(G_0) = V(G)$.
\end{lma}

\begin{proof}
We fix $\ell$ and proceed by induction on $\ell'$.
The statement is trivially true for $\ell' =0$ so we may assume that $\ell' >0$.
Let $f_i = u_iw_i$ for $1 \le i \le \ell'$.
By Corollary~\ref{rotation2}(i) (with $G_{\ell} = G_L$), we know that $w \in  N_G(w_{\ell'})$.
Since $w_{\ell'} \notin \{x,y\}$ and $G$ is a $1$-path-cycle, we have $d_G(w_{\ell'}) = 2$.
Let 
\begin{align}
N_G(w_{\ell'}) = \{w, w'\}. \label{eqn:w'}
\end{align}
There exists a $1$-path-cycle $G'_L$ with parameters $(u_{\ell'},c_{u_{\ell'}};y, c_y)$ obtained from $G$ by $\ell'-1$ rotations using the chord sequence $f_1, \dots,f_{\ell'-1}$.
Moreover, $G_L$ can be obtained from $G'_L$ by a rotation using $f_{\ell'}$.
Since $f_1, \dots,f_{\ell'}$ is spread out in~$G$, $\dist_G(w_{\ell'} , v) > 5$ for all $v \in \{ x, y ,w_i : 1 \le i < \ell'\}$.
Hence, Corollary~\ref{rotation2}(ii) and \eqref{eqn:w'}  imply that 
\begin{align*}
N_{G'_L}(w_{\ell'}) = N_G(w_{\ell'}) = \{w, w'\}.
\end{align*}
Furthermore, since $G_L$ can be obtained from $G'_L$ by a left rotation using the chord $f_{\ell'} = u_{\ell'}w_{\ell'}$, Lemma~\ref{rotation}(iii) (with $G = G'_L$, $G' = G_L$ and left chord $xw = f_{\ell'}$) implies that
\begin{align}
c(uw_{\ell'}) \ne c(w_{\ell'}w'). \label{c(uw')}
\end{align}

On the other hand, recall that $G'_L$ has parameters $(u_{\ell'},c_{u_{\ell'}};y, c_y)$.
By the induction hypothesis, there exists a $1$-path-cycle $G'$ with parameters $(u_{\ell'},c_{u_{\ell'}} ; z, c_{z})$ obtained from $G$ by rotations using $e_1, \dots, e_{\ell} , f_1, \dots, f_{\ell'-1}$.
Let $e_i = u'_iw'_i$ for $1 \le i \le \ell$.
Since $e_1, \dots, e_{\ell} , f_1, \dots, f_{\ell'}$ is spread out in~$G$,
% , we have $\dist(w_{\ell'} , v) > 5$ for all $v \in \{ x, y ,w_i,w'_j : 1 \le i < \ell', 1 \le j \le \ell\}$.
we have $w_{\ell'} \notin N_{G'}(u_{\ell'}) \cup N_{G'}(z)\cup \{u_{\ell'},z\}$ and $N_{G'}(w_{\ell'}) = N_G(w_{\ell'}) = \{w,w'\}$ by Corollary~\ref{rotation2}(ii) and \eqref{eqn:w'}.
Hence $f_{\ell'}$ is also a left chord for $G'$.
By \eqref{c(uw')} and Lemma~\ref{rotation}(iii) (with $G = G'$ and left chord $xw = f_{\ell'}$), there exists a $1$-path-cycle $G_0$ with parameters $(w,c_{w} ; z, c_{z})$ obtained from $G'$ by a rotation using $f_{\ell'}$.
This completes the proof of the lemma.
\end{proof}

Let $G$ be a $1$-path-cycle in $K_n^c$ with $|G|$ maximal.
Further suppose that $G$ has parameters $(x,c_x ; y, c_y)$.
In the next lemma, we show that there exists a vertex $z \in V(G)$ and two distinct colours $c^1$ and $c^2$ such that for $i=1,2$ there exists $1$-path-cycle $G^i$ with parameters $(x,c_x;z,c^i)$ obtained from $G$ by right rotations only.

\begin{lma} \label{lma:rotation3}
Let $0 < \eps < 1/2$.
Then there exists an integer $n_0$ such that whenever $n \ge n_0$ the following holds.
Suppose that $K_n^c$ is an edge-coloured $K_n$ with $\Dmon(K_n^c) \le (1/2 - \eps)n $.
Let $G$ be a properly coloured $1$-path-cycle in $K_n^c$ with $|G|$ maximal.
Suppose that $G$ has parameters $(x,c_x ; y, c_y)$.
Let $U$ be a subset of $V(K_n^c) \setminus \{x,y\}$ of size at most $\eps n /8$.
Then there exist an integer $1 \le \ell \le \lceil 1 / \log_2(1+ \eps) \rceil + 1$ and a vertex $z \in V(G) \setminus U$ such that
\begin{itemize}
	\item[(a)] for each $i=1,2$, there exists a $1$-path-cycle $G^i$ with parameters $(x,c_x;z,c^i_z)$ obtained from $G$ by $\ell$ right rotations using a chord sequence $e^i_1, \dots,e^i_{\ell}$ and $V(G^i) = V(G)$;
	\item[(b)] for $i=1,2$, the chord sequence $e^i_1, \dots,e^i_{\ell}$ is spread out in $G$;
	\item[(c)] $c^1_z \ne c^{2}_z$;
	\item[(d)] $V(e^i_j) \subseteq V(G) \setminus U$ for all $i \le 2$ and all $j \le \ell$.
\end{itemize}
Moreover, the similar statements hold for left rotations.
\end{lma}

Note that the two chord sequences $e^1_1, \dots,e^1_{\ell}$ and $e^2_1, \dots,e^2_{\ell}$ obtained from the lemma above are not necessarily vertex-disjoint from each other.
The key ingredient of the proof is the set $Z_{\ell}$ of pairs $(z,c_z)$ for $z \in V(G) \setminus U$ and colours $c_z$ such that 
\begin{itemize}
	\item[\rm (i)] there exists a $1$-path-cycle $G^{(z,c_z)}_{\ell}$ obtained from $G$ by $\ell$ right rotations using a chord sequence $e_1, \dots,e_{\ell}$;
	\item[\rm (ii)] $G^{(z,c_z)}_{\ell}$ has parameter $(x,c_x;z,c_{z})$;
	\item[\rm (iii)] the chord sequence $e_1, \dots,e_{\ell}$ is spread out in $G$;
	\item[\rm (iv)] $V(e_j) \subseteq V(G) \setminus U$ for all $j \le \ell$.
\end{itemize}
If there exist $(z, c_z), (z, c'_z) \in Z_{\ell}$ with $c_z \ne c'_z$ for some $z \in V(G) \setminus U$ and some~$\ell$, then the lemma holds.
Otherwise, we may assume that each $z \in V(G) \setminus U$ `appears' at most once in each $Z_{\ell}$.
We then show that $|Z_{\ell}| \ge (1+ \eps)^{\ell}n/2$ for $\ell \ge 1$.
Since $|Z_{\ell+1}|$ is bounded above by~$n$, we obtain a contradiction provided $\ell$ is large enough.

\begin{proof}[Proof of Lemma~$\ref{lma:rotation3}$]
Let $n_0 = \left\lceil 11\eps^{-1} \left( \lceil 1 / \log_2(1+ \eps) \rceil +3 \right) \right\rceil$.
Let $K_n^c$, $G$ and $U$ be as defined in the lemma.
For integers $\ell \ge 0$, define $Z_{\ell}$ as above. 
% let $Z_\ell$ be the set of pairs $(z,c_z)$ such that $z \in V(G) \setminus U$ and there exists a $1$-path-cycle $G^{(z,c_z)}_{\ell}$ with parameters $(x,c_x;z,c_{z})$ obtained from $G$ by $\ell$ right rotations using chords $e_1, \dots,e_{\ell}$ with $V(e_j) \subseteq V(G) \setminus U$ for all $j \le \ell$ and the chords spread out in $G$.
Thus, $Z_0 = \{ (y, c_y )\}$.
To prove the lemma, it is enough to show that there exist $z \in V(G) \setminus U$ and an integer $\ell \le  \lceil 1 / \log_2(1+ \eps) \rceil+1$ such that $(z, c_z), (z, c'_z) \in Z_{\ell}$ with $c_z \ne c'_z$.

Suppose the lemma is false.
Hence, for each integer $1 \le \ell \le  \lceil 1 / \log_2(1+ \eps) \rceil+1$, if $(z,c_z) \in Z_\ell$, then $c_z$ is uniquely determined by~$z$ and~$\ell$ (or else we are done).
We simply write $\mathbf{z} \in Z_\ell$ for $(z,c_z) \in Z_\ell$.
Note that
\begin{align}
|Z_\ell| \le n \text{ for all }0 \le \ell \le \lceil 1 / \log_2(1+ \eps) \rceil+1. \label{|Z_i|}
\end{align}
For each $\mathbf{z} \in Z_\ell$, we fix a $1$-path-cycle $G_{\ell}^{\mathbf{z}}$ and a chord sequence $e_1^{\mathbf{z}}, \dots, e_\ell^{\mathbf{z}}$ such that
\begin{itemize}
	\item[\rm (i$'$)] $G_{\ell}^{\mathbf{z}}$ is obtained from $G$ by $\ell$ right rotations using the chord sequence $e_1^{\mathbf{z}}, \dots, e_\ell^{\mathbf{z}}$;
	\item[\rm (ii$'$)] $G^{\mathbf{z}}_{\ell}$ has parameter $(x,c_x;z,c_{z})$;
	\item[\rm (iii$'$)] the chord sequence $e_1^{\mathbf{z}}, \dots, e_\ell^{\mathbf{z}}$ is spread out in $G$;
	\item[\rm (iv$'$)] $V(e_j^{\mathbf{z}}) \subseteq V(G) \setminus U$ for all $j \le \ell$.
\end{itemize}
We denote by $P_{\ell}^{\mathbf{z}}$ the path in $G_{\ell}^{\mathbf{z}}$.
Recall that $V(G^{\mathbf{z}}_\ell) = V(G)$ by Corollary~\ref{rotation2}(i).
For every $v \in V(K_n^c) \setminus V(G)$, we have $c(vz) = c_z$.
Otherwise, we can extend $P_{\ell}^{\mathbf{z}}$ enlarging the $1$-path-cycle~$G_{\ell}^{\mathbf{z}}$, which contradicts the maximality of~$|G|$.
Since $\Dmon (K_n^c)  \le (1/2 - \eps)n$, for each $\mathbf{z} =(z, c_z) \in Z_\ell$, we have
\begin{align}
|\{ v  \in V(G) \setminus z : c(vz) \ne c_z \}| & = |\{ v  \in V(K^c_n) \setminus z : c(vz) \ne c_z \}| \nonumber \\
& \ge n-1 - \Dmon \ge \left( 1/2+\eps \right) n -1.
 \label{dmon}
\end{align}

Set $U' = U \cup \bigcup_{u \in U } N_G(u)$ and $V' = V(G) \setminus U'$.
So  $|U'| \le 3|U| \le 3\eps n/8$.
For each integer $0 \le \ell \le  \lceil 1 / \log_2(1+ \eps) \rceil$, define an auxiliary bipartite graph~$H_\ell$ with vertex classes $Z_\ell$ and $V'$, and the edge set $E(H_\ell)$ satisfies for all $\mathbf{z} =(z,c_z)\in Z_\ell$ and all $v \in V'$, $\mathbf{z} v$ is an edge in $H_\ell$ if and only if $zv$ is a (right) chord for $G^{\mathbf{z}}_{\ell}$ and the chord sequence $e^{\mathbf{z}}_1, \dots,e^{\mathbf{z}}_{\ell}, zv $ is spread out in~$G$.
Given $0 \le \ell \le  \lceil 1 / \log_2(1+ \eps) \rceil$ and $\mathbf{z} = (z,c_z) \in Z_\ell$, note that the number of vertices $v$ such that $e^{\mathbf{z}}_1, \dots,e^{\mathbf{z}}_{\ell}, zv$ is not spread out is at most $11 ( \ell + 2)$.
Recall that $zv$ is a chord for $G^{\mathbf{z}}_{\ell}$ if $z \ne v \ne x$ and $c(vz) \ne c_z$.
By the definition of $H_\ell$ and \eqref{dmon}, for each $\mathbf{z} = (z,c_z) \in Z_\ell$,
\begin{align}
 d_{H_\ell}(\mathbf{z} ) & \ge |\{ v  \in V(G) : c(vz) \ne c_z \}| -|U'| - 11(\ell+2) \nonumber \\
  & \ge \left( \frac12+\eps \right) n -1 - \frac{3 \eps n}{8}- 11\left( \left\lceil \frac{1}{\log_2(1+ \eps)} \right\rceil+2 \right) \nonumber \\
& \ge (1 + \eps) n/2 \label{degz}
\end{align}
as $n$ is large.
Hence,
\begin{align}
 e(H_\ell) \ge \sum_{\mathbf{z} \in Z_\ell} d_{H_\ell}(\mathbf{z})\ge (1 + \eps ) |Z_\ell| n/2. \label{e(Hi)lower}
\end{align}

Next we investigate how $E(H_\ell)$ and $Z_{\ell+1}$ are related.
Suppose that $\mathbf{z} v$ is an edge in $H_\ell$ with $\mathbf{z} \in Z_\ell$ and $v \in V'$.
Since $zv$ is a right chord for $G^{\mathbf{z}}_\ell$, by Lemma~\ref{rotation} we know that there exists a $1$-path-cycle $G'$ with parameters $(x,c_x; v', c_{v'})$ obtained from $G^{\mathbf{z}}_\ell$ by a rotation using $zv$.
This means that $G'$ can be obtained from $G$ by rotations using $e^{\mathbf{z}}_1, \dots,e^{\mathbf{z}}_{\ell}, zv$.
Since $\mathbf{z} v$ is an edge in $H_\ell$, the chord sequence $e^{\mathbf{z}}_1, \dots,e^{\mathbf{z}}_{\ell}, zv$ is spread out in $G$. 
Recall that $v \notin U' = U \cup \bigcup_{u \in U} N_G(u)$.
Corollary~\ref{rotation2}(i) implies that $v' \in N_G(v)$ and so $v' \notin U$.
Therefore, $( v', c_{v'}) \in Z_{\ell+1}$.
So this gives a natural map $\phi$ from $e(H_{\ell})$ to~$Z_{\ell+1}$, namely $\phi( \mathbf{z} v ) = (v',c_{v'})$.
Note that $N_G(v') = \{ v , v''\}$ and $c_{v'} = c(v'v'')$ by Lemma~\ref{rotation}(ii).
Recall that if $(z',c_{z'}) \in Z_{\ell+1}$, then $c_{z'}$ are be uniquely determined by $c_{z'}$ and~$\ell+1$.
Therefore, 
\begin{align}
\text{if $\mathbf{z}v, \mathbf{z'}v' \in e(H_\ell)$ with $v \ne v'$, then $\phi(\mathbf{z}v) \ne \phi(\mathbf{z'}v')$.} \label{phi}
\end{align}
So $|Z_{\ell+1}| \ge | \bigcup_{\mathbf{z} \in  Z_\ell} N_{H_\ell} (\mathbf{z}) |$.
Since $Z_0 = \{(y,c_y)\} $, by~\eqref{degz} we have 
\begin{align}
|Z_1| \ge d_{H_0}( (y,c_y) )  \ge (1+ \eps) n/2. \label{Z1}
\end{align}

Now suppose that $1 \le \ell \le \lceil 1 / \log_2(1+ \eps) \rceil$.
We edge-colour $H_\ell$ such that the edge $\mathbf{z}v$ in $H_\ell$ has the colour~$c(zv)$ (appeared in~$K_n^c$).
Let $X_\ell$ be the set of vertices in $V'$ that see exactly one colour in $H_\ell$.
Let $Y_\ell$ be the set of vertices in $V'$ that see at least $2$ colours in $H_\ell$.
Given $v \in Y_\ell$, there exist $\mathbf{z}_1, \mathbf{z}_2 \in Z_\ell$ such that $\mathbf{z}_1v, \mathbf{z}_2v \in E(H_\ell)$ and $c(z_1v) \ne c(z_2v)$.
Let $N_G(v) = \{v_1,v_2\}$.
Without loss of generality, we may assume that $c( z_1 v ) \ne c( v v_2 )$ and $c( z_2 v ) \ne c( v v_1 )$.
By Lemma~\ref{rotation}(ii) and~(iii), there exists a 1-path cycle $G'$ with parameters $(x,c_x; v_1, c(vv_2) )$ obtained from $G^{\mathbf{z}_1}_\ell$ by rotations using $z_1v$.
Hence, $(v_1, c( v v_2)) \in Z_{\ell+1}$ and similarly $(v_2, c( v v_1 )) \in Z_{\ell+1}$.
In summary, every $y \in Y_\ell$ contributes to two distinct members of $Z_{\ell+1}$ and every $x \in X_\ell$ contributes to at least one member of $Z_{\ell+1}$.
Moreover, by a similar argument used to prove~\eqref{phi}, we deduce that all members of $Z_{\ell+1}$ derived this way are distinct.
This means that for $1 \le \ell \le  \lceil 1 / \log_2(1+ \eps) \rceil$, we have
\begin{align}
		|Z_{\ell+1}| \ge |X_\ell| + 2 |Y_\ell|. \label{zi+1}
\end{align}
Since each vertex $x \in X_\ell$ meets edges of the same colour in~$H_\ell$, $d_{H_\ell}(x) \le \Dmon (K_n^c) \le n/2$.
By counting the degrees of $w \in X_\ell  \cup Y_\ell \subseteq V'$ in $H_\ell$, we have
\begin{align}
	e(H_\ell) = \sum_{w \in X_\ell \cup Y_\ell} d_{H_\ell} (w) \le  |X_\ell|n/2 +  |Y_\ell|n \le |Z_{\ell+1}| n/2, \nonumber
\end{align}
where the last inequality is due to~\eqref{zi+1}.
Together with~\eqref{e(Hi)lower}, we have
\begin{align}
		|Z_{\ell+1}| & \ge 2e(H_\ell)/n \ge  (1+ \eps) |Z_\ell| \nonumber
\end{align}
for all $1 \le \ell\le \lceil \log_2(1+ \eps) \rceil $.
Therefore, 
\begin{align}
	|Z_{\ell+1}| & \ge (1+\eps)^\ell |Z_1| \ge (1 + \eps)^{\ell+1} n/2 \nonumber
\end{align}
for all $0 \le \ell \le \lceil 1 / \log_2(1+ \eps) \rceil$, where the last inequality is due to~\eqref{Z1}.
This implies that $|Z_\ell| > n$ when $\ell = \lceil 1 / \log_2(1+ \eps) \rceil+1$, contradicting~\eqref{|Z_i|}.
\end{proof}

We are ready to prove Lemma~\ref{lma:2factor} using Lemmas~\ref{lma:rotation4}  and~\ref{lma:rotation3}.

\begin{proof}[Proof of Lemma~$\ref{lma:2factor}$]

Let $n_0$ be the integer given by Lemma~\ref{lma:rotation3} and let 
\begin{align*}
n_1 =   \max \left\{ \left\lceil \frac{1000}{\eps\log_2(1+ \eps)}  \right\rceil, n_0 \right\} .
\end{align*}
Let $K_n^c$ be an edge-coloured $K_n$ with $\Dmon(K_n^c) \le (1/2 - \eps)n $ and $n \ge n_1$.
Let $G$ be a properly coloured $1$-path-cycle in $K_n^c$ with $|G|$ maximal.
We may assume that $G$ is not a $2$-factor or else we are done.
By applying Theorem~\ref{Feng} to $K_n^c[V(G)]$, we may assume that $G$ is a properly coloured path. 
Hence, $G$ is a $1$-path-cycle with parameter $(x,c_x;y,c_y)$ with $x \ne y$.
Apply Lemma~\ref{lma:rotation3} (with $U = \emptyset$) to $G$ and obtain an integer $\ell \le \lceil 1 / \log_2(1+ \eps) \rceil +1$, a vertex $ z \in V(G)$ and two chord sequences $e^1_1, \dots,e^1_{\ell}$ and  $e^2_1, \dots,e^2_{\ell}$ such that 
\begin{itemize}
	\item[(a)] for $i =1,2$, there exists a $1$-path-cycle $G^i_R$ with parameters $(x,c_x;z,c^i_z)$ obtained from $G$ by $\ell$ right rotations using $e^i_1, \dots,e^i_{\ell}$ such that $V(G^i_R) = V(G)$;
	\item[(b)] for $i =1,2$, the chord sequence $e^i_1, \dots,e^i_{\ell}$ is spread out in $G$;
	\item[(c)] $c^1_z \ne c^{2}_z$.
\end{itemize}
Let $U'$ be the set of vertices $u \in V(G)$ such that $\dist_G(u,v) \le 5$ for some $v \in \{x,y\} \cup \bigcup_{i,j} V(e^i_j)$.
Let $U = U' \setminus \{x,y\}$.
Hence $|U| \le 11(2+ 4\ell ) \le \eps n /8$ since $n$ is large.
By the left rotation version of Lemma~\ref{lma:rotation3}, there exist an integer $\ell' \le \lceil 1 / \log_2(1+ \eps) \rceil+1$, a vertex $ w \in V(G) \setminus U$ and two chord sequences $f^1_1, \dots,f^1_{\ell'}$ and $f^2_1, \dots,f^2_{\ell'}$ such that for $i =1,2$
\begin{itemize}
	\item[(a$'$)] for $i =1,2$, there exists a $1$-path-cycle $G^i_L$ with parameters $(w,c^{i}_w;y,c_y)$ obtained from $G$ by $\ell'$ left rotations using $f^i_1, \dots,f^i_{\ell'}$ such that $V(G^i_L) = V(G)$;
	\item[(b$'$)] for $i =1,2$, the chord sequence $f^i_1, \dots,f^i_{\ell'}$ is spread out in $G$;
	\item[(c$'$)] $c^1_w \ne c^{2}_w$;
	\item[(d$'$)] $V(f^i_j) \subseteq V(G) \setminus U$ for all $i \le 2$ and all $j \le \ell'$.
\end{itemize}
By (c) and (c$'$), we may assume without loss of generality that $c_z^1 \ne c(zw) \ne c_w^1$.
Note that the sequence $e^1_1, \dots,e^1_{\ell},f^1_1, \dots,f^1_{\ell'}$ is spread out in $G$ by (b), (b$'$), (d$'$) and the definition of~$U$.
Apply Lemma~\ref{lma:rotation4} (with $G_L = G^1_L$ and $G_R = G^1_R$) and obtain a $1$-path-cycle~$G_0$ with parameters $(w,c^{1}_w;z,c^1_z)$.
Since $c_z^1 \ne c(zw) \ne c_w^1$, $G_0 + zw$ is a union of vertex-disjoint properly coloured cycles with $V(G_0+ zw) = V(G_0) = V(G)$. 
If $V(K_n^c) \ne V(G_0)$, then $G_0 + zw$ together with a vertex $v \in V(K_n^c) \setminus V(G)$ is also properly coloured $1$-path-cycle, contradicting the maximality of $|G|$.
Therefore, $V(G_0) = V(K_n^c)$ implying that $G_0 + zw$ is a properly coloured $2$-factor in $K_n^c$ as required.
\end{proof}

%%%%%%%%%%%%%%%%%%%%%%%%%%%%%%%%%%%%%%%%%%%%%%%%%%%%%%%%%%%%%%%%%%%%%%%%%%%%%%%%%%

\section{Proof of Theorem~\ref{PCHC}} \label{sec:proof}
We may assume that $0 < \eps <1/4$.
Let $n_0$ be the integer given by Lemma~\ref{lma:abscycle}.
Set $\gamma = 2^{-5} \eps^{4 \eps^{-2} +2}$ and $\eps' = (2\eps - \gamma)/(2-2\gamma)$.
Note that $0 <\eps' < 1/2$.
Let $N_0 = \max \{ n_0,  \lceil n_1 (\eps')/(1- \gamma) \rceil  \}$, where $n_1$ is the function given by Lemma~\ref{lma:2factor}.

Let $K_n^c$ be an edge-coloured complete graph on $n$ vertices with $\Dmon(K_n^c) \le (1/2 - \eps) n$ and $n \ge N_0$.
Let $C$ be the properly coloured cycle given by Lemma~\ref{lma:abscycle} and so $|C| \le \gamma n $.
Let $K_{n'}^c = K_n^c \setminus V(C)$.
Note that 
\begin{align}
	\Dmon( K_{n'}^c) \le (1/2 - \eps) n \le (1/2 - \eps') n', \nonumber
\end{align}
where $n' = n - |C| \ge ( 1- \gamma) n \ge n_1 (\eps')$.
There exists a properly coloured $2$-factor in $K_{n'}^c$ by Lemma~\ref{lma:2factor}.
Hence, $K_{n'}^c$ contains a properly coloured Hamiltonian path~$P$ by Theorem~\ref{Bang}.
By the property of $C$ guaranteed by Lemma~\ref{lma:abscycle}, there exists a properly coloured cycle $C'$ spanning the vertex set $V(C) \cup V(P) = V(K_n^c)$.
Hence, $C'$ is a properly coloured Hamiltonian cycle in $K_n^c$.
$\hfill{\square}$

%%%%%%%%%%%%%%%%%%%%%%%%%%%%%%%%%%%%%

\section{Acknowledgements}

The author would like to thank Timothy Townsend for his comments on a draft of this paper.

% \bibliographystyle{amsplain}
% \bibliography{alternatinghamiltonian}

\end{document}